\documentclass[11pt]{article}

\usepackage{psfrag}
\usepackage{amsmath}
\usepackage{amssymb}
\usepackage{amscd}
\usepackage{picinpar}
\usepackage{times}
\usepackage{pb-diagram}
\usepackage{graphicx}
\usepackage{wrapfig}
\usepackage{xspace}

\def\R{\mathbb{R}}
\def\H{\mathbf{H}}
\def\T{\mathcal{T}}
\def\A{\mathbf{A}}
\def\B{\mathbf{B}}

\def\D{\mathcal{D}}

\newtheorem{theorem}{Theorem}[section]
\newtheorem{lemma}[theorem]{Lemma}

\newtheorem{definition}[theorem]{Definition}

\newenvironment{proof}[1][Proof]{\begin{trivlist}
\item[\hskip \labelsep {\bfseries #1}]}{\end{trivlist}}
\newcommand{\qed}{\nobreak \ifvmode \relax \else
      \ifdim\lastskip<1.5em \hskip-\lastskip
      \hskip1.5em plus0em minus0.5em \fi \nobreak
      \vrule height0.75em width0.5em depth0.25em\fi}

\begin{document}

\title{A discrete uniformization theorem
for polyhedral surfaces}

\author{ Xianfeng Gu, Feng Luo, Jian Sun, Tianqi Wu}

\date{}

\maketitle

\begin{abstract} A discrete conformality for
polyhedral metrics on surfaces is introduced in this paper which
generalizes earlier work on the subject. 
 It is shown that each polyhedral metric on a surface is 
  discrete
conformal to a constant curvature polyhedral metric which is
unique up to scaling. Furthermore, the constant curvature metric
can be found using a discrete Yamabe flow with surgery.
\end{abstract}

\section{Introduction}
\subsection{Statement of results} The Poincare-Koebe uniformization theorem for Riemann
surfaces is a pillar in the last century mathematics. It states
that given any Riemannian metric on a connected surface, there
exists a complete constant curvature Riemannian metric conformal
to the given one. Furthermore, the complete metric of curvature -1
is unique unless the underlying Riemann surface is biholomorphic
to the Riemann sphere, a torus, or the punctured plane. The
uniformization theorem has a wide range of applications within and
outside mathematics. There have been much work on establishing
various discrete versions of the uniformization theorem for
discrete or polyhedral surfaces. A key step in discretization is
to define the concept of discrete conformality. The most prominent
one is probably Thurston's circle packing theory. The purpose of
this paper is to introduce a discrete conformality for polyhedral
metrics and discrete Riemann surfaces and establish a discrete
uniformization theorem within the category of polyhedral metrics
(PL metrics) on compact surfaces.

Polyhedral surfaces are ubiquitous in computer graphics and many
fields of sciences nowadays. Organizing polyhedral surfaces
according to their conformal classes is a very useful and
important principle. However, to decide if two polyhedral surfaces
are conformal in the classical (Riemannian)
sense is highly non-trivial and time consuming. 
The discrete conformality introduced in this paper overcomes this
computational difficulty. 

Given a closed surface $S$ and a finite non-empty set $V \subset
S$, we call $(S, V)$ a \it marked surface\rm.    The objects of
our investigation are \it polyhedral metrics \rm (or simply PL
metrics) on surfaces.  By definition, a PL metric on $(S, V)$ is a
flat cone metric on $S$ whose cone points are in $V$. For
instance, the boundary of a tetrahedron in the 3-space is a PL
metric on the 2-sphere with 4 cone points. The norms of
holomorphic quadratic differentials on Riemann surfaces are other
examples of PL metrics. The \it discrete curvature \rm of a PL
metric on $(S,V)$ is the function on $V$ sending a vertex $v \in
V$ to $2\pi$ less the cone angle at $v$. A triangulation $\T$ of
$S$ with vertex set $V$ is called a \it triangulation \rm of
$(S,V)$.
Each PL metric $d$ on $(S, V)$ has a \it Delaunay triangulation
$\T(d)$ \rm of $(S,V)$ so that each triangle in $\T(d)$ is
Euclidean and the sum of two angles facing each edge is at most
$\pi$.


\begin{definition} \label{dc}(Discrete conformality and discrete Riemann surface)
 Two PL metrics $d, d'$ on $(S,V)$ are discrete conformal  if there exist sequences
 of PL
metrics $d_1=d, ..., d_m =d'$ on $(S, V)$ and triangulations
$\T_1, ..., \T_m$ of $(S, V)$  satisfying

(a) each $\T_i$ is Delaunay in $d_i$,

(b) if $\T_i=T_{i+1}$, there exists a function $u:V \to \R$,
called a conformal factor,  so that if $e$ is an edge in $\T_i$
with end points $v$ and $v'$, then the lengths $l_{d_{i+1}}(e)$
and $l_{d_i}(e)$ of $e$ in $d_i$ and $d_{i+1}$ are related by
\begin{equation}
\label{conformal} l_{d_{i+1}}(e) =l_{d_i}(e) e^{ u(v)+u(v')},
\end{equation}

(c) if $\T_i \neq \T_{i+1}$, then $(S, d_i)$ is isometric to $(S,
d_{i+1})$ by an isometry homotopic to the identity in $(S, V)$.

The discrete conformal class of a PL metric is called a \it
discrete Riemann surface.\rm
\end{definition}

\begin{figure}[ht!]
\begin{center}
\begin{tabular}{c}
\includegraphics[width=0.9\textwidth]{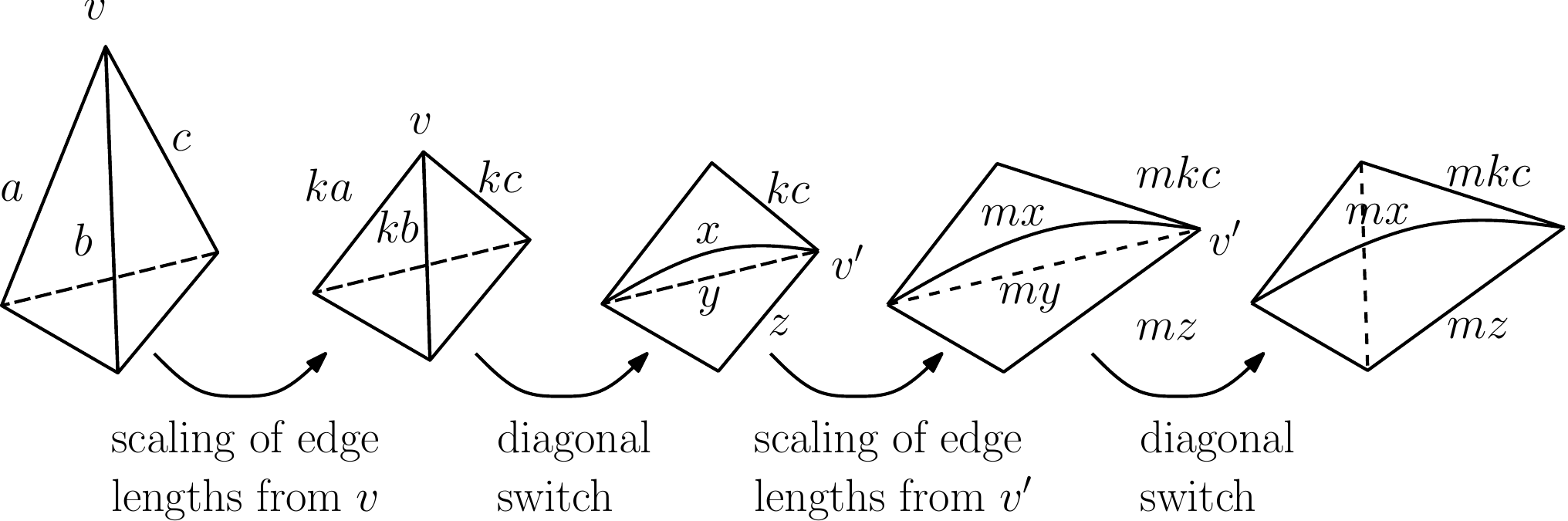}
\end{tabular}
\end{center}
\caption{discrete conformal change of PL metrics, all
triangulations are Delaunay} \label{figure_1}
\end{figure}


\begin{theorem} \label{main} Suppose $(S, V)$ is a closed connected marked surface
and  $d$ is any PL metric on $(S, V)$.  Then for any $K^*:V \to
(-\infty, 2\pi)$ with $\sum_{v \in V} K^*(v) =2\pi \chi(S)$, there
exists a PL metric $d'$, unique up to scaling,  on $(S, V)$ so
that $d'$ is discrete conformal to $d$ and the discrete curvature
of $d'$ is $K^*$. Furthermore, the discrete Yamabe flow with
surgery associated to curvature $K^*$ with initial value $d$
converges to $d'$ exponentially fast.
\end{theorem}

For the constant function $K^*=2\pi\chi(S)/|V|$ in theorem
\ref{main}, we obtain a constant curvature PL metric $d'$, unique
up to scaling, discrete conformal to $d$. This is a discrete
version of the uniformization theorem. Theorem \ref{main} also
holds for compact marked surfaces with non-empty boundary. In that
case, we double the surface to obtain a closed surface. We omit
the details.


The prototype of definition 1.1 comes from the work of
Ro$\check{c}$ek and Williams in physics \cite{rw} and \cite{luo}.
The drawback of the definition in \cite{rw} and \cite{luo} is that
it depends on the choice of triangulations.  A convex variational
principle associated to the discrete conformality was established
in \cite{luo}.

It is highly desirable to have a quantitative estimate of the
difference between discrete conformality and classical
conformality. See \cite{gls} for an estimate of this type.

 There are many
proofs of the Poincare-Koebe uniformization theorem. The proof
most closely related to our work is Hamilton's Ricci flow. The
Ricci flow proof of the uniformization theorem for closed surfaces
was achieved by a combination of the work of \cite{hamilton},
\cite{chow}, and \cite{chenlutian}. 
 In the discrete case, the situation is much more
complicated due to the combinatorics. To prove theorem \ref{main},
we use Penner's decorated Teichumuller theory \cite{penner}, the
work of Bobenko-Pinkall-Springborn \cite{bps} relating PL metrics
to Penner's theory and a variational principle developed in
\cite{luo}.

Hamilton's Ricci flow is a flow in the space of all Riemannian
metrics on a manifold. In the discrete setting, the discrete
Yamabe flow with surgery is a $C^1$-smooth flow on the finite
dimensional Teichm\"uller space of flat cone metrics on a closed
marked surface $(S,V)$.  

A theorem of Troyanov \cite{troyanov}  states that the same result
of theorem \ref{main} holds if discrete conformality is replaced
by
the classical Riemannian conformality. 
The major difference between Troyanov's work and  theorem
\ref{main} is that in our case, we discretize the metric and
conformality so that a metric is represented as a edge length
vector in $\R^N$ and discrete conformality can be decided
algorithmically from edge length vector. Theorem \ref{main} is
also related to the work of Kazdan and Warner \cite{kw1} and
\cite{kw2} on prescribing Gaussian curvature. It is possible that
theorem \ref{main} implies the existence part of Troyanov's
theorem and Kazdan-Warner's theorem for closed surfaces by
approximation.

The similar theorem for hyperbolic cone metrics on $(S,V)$ has
been proved in \cite{gglsw}. In this case, two hyperbolic cone
metrics $d, d'$ on $(S,V)$ are \it discrete conformal \rm if there
exist sequences of hyperbolic cone metrics $d_1=d, ..., d_m =d'$
on $(S, V)$ and triangulations $\T_1, ..., \T_m$ of $(S, V)$
satisfying (a) each $\T_i$ is Delaunay in $d_i$, and (b) if
$\T_i=T_{i+1}$, there exists a function $u:V \to \R$ so that if
$e$ is an edge in $\T_i$ with end points $v$ and $v'$, then the
lengths $l_{d_{i+1}}(e)$ and $l_{d_i}(e)$ of $e$ in $d_i$ and
$d_{i+1}$ are related by
\begin{equation}\label{hy} \sinh(\frac{ l_{d_{i+1}}(e)}{2})
=\sinh(\frac{l_{d_i}(e)}{2}) e^{ u(v)+u(v')}, \end{equation} and
(c) if $\T_i \neq \T_{i+1}$, then $(S, d_i)$ is isometric to $(S,
d_{i+1})$ by an isometry homotopic to the identity in $(S, V)$.
The condition (\ref{hy}) was first introduced in \cite{bps}.

\begin{theorem}  Suppose $(S, V)$ is a closed connected marked surface
and  $d$ is any hyperbolic cone metric on $(S, V)$.  Then for any
$K^*:V \to (-\infty, 2\pi)$ with $\sum_{v \in V} K^*(v) >2\pi
\chi(S)$, there exists a unique hyperbolic cone metric $d'$ on
$(S, V)$ so that $d'$ is discrete conformal to $d$ and the
discrete curvature of $d'$ is $K^*$. Furthermore, the discrete
Yamabe flow with surgery associated to curvature $K^*$ with
initial value $d$ converges to $d'$ exponentially fast. In
particular, if $\chi(S)<0$ and  $K^*=0$, each hyperbolic cone
metric on $(S,V)$ is discrete conformal to a unique hyperbolic
metric on $S$.
\end{theorem}

\subsection{Notations and conventions}

Triangulations to be used in the paper are defined as follows.
Take a finite disjoint union of Euclidean triangles and identify
edges in pairs by homeomorphisms. The quotient space is a compact
surface together with a \it triangulation \rm $\T$ whose simplices
are the quotients of the simplices in the disjoint union. Let
$V=V(\T)$ and $E=E(\T)$ be the sets of vertices and edges in $\T$.
 If $e$ is an edge in
$\T$ adjacent to two distinct triangles $t, t'$, then the \it
diagonal switch \rm on $\T$ at $e$ replaces $e$ by the other
diagonal in the quadrilateral $t \cup_e t'$ and produces a new
triangulation $\T'$ on $(S,V)$. A PL metric $d$ on $(S,V)$ is
obtained as isometric gluing of Euclidean triangles along edges so
that the set of cone points is in $V$. Given a PL metric $d$ and a
triangulation $\T$ on $(S, V)$, if each triangle in $\T$ (in $d$
metric) is isometric to a Euclidean triangle, we say $\T$ is \it
geometric \rm in $d$.  If $\T$ is a triangulation of $(S,V)$
isotopic to a geometric triangulation $\T'$ in a PL metric $d$,
then the \it length \rm of an edge  $e \in E(\T)$ (or \it angle
\rm of a triangle at a vertex in $\T$) is defined to be the length
 of the corresponding geodesic edge $e' \in E
(\T')$ (respectively angle of the corresponding triangle in $\T'$)
measured in metric $d$. The interior of a surface $X$ is denoted
by $int(X)$. If $X$ is a finite set, $|X|$ denotes its cardinality
and $\R^X$ denotes the vector space $\{f: X \to \R\}$. All
surfaces are assumed to be connected.

\subsection{Acknowledgement} The work is supported in part by the NSF of USA and
the NSF of China.

\section{Teichm\"uller space of PL metrics and Delaunay conditions}

Suppose $(S, V)$ is a marked connected surface. The discrete
curvature $K: V \to (-\infty, 2\pi)$ of a PL metric $d$ on $S$
satisfies the Gauss-Bonnet formula that $\sum_{v \in V} K(v) =2\pi
\chi(S)$. Therefore, if $\chi(S-V) \geq 0$, i.e., $(S, V) =(S^2,
\{v_1,...,v_n\})$ with $n \leq 2$, the Gauss-Bonnet identity
implies there is no PL metric on $(S, V)$. From now on, we will
always assume that the Euler characteristic $\chi(S-V) <0$. Most
of the results in this section are well known. We omit details.

\subsection{Teichm\"uller space of PL metrics and its length coordinates}

Two PL metrics $d, d'$ on $(S, V)$ are called \it equivalent \rm
if there is an isometry $h: (S, V, d) \to (S, V, d')$ so that $h$
is isotopic to the identity map on $(S, V)$. The \it Teichm\"uller
space of all PL metrics \rm on $\Sigma$, denoted by $T_{pl}(S,V)$,
is the set of all equivalence classes of PL metrics on $(S,V)$,
i.e.,
$$T_{pl} = T_{pl}(S,V)=\{ d |\text{ $d$ is a PL metric on $(S, V)$} \}/
isometry \cong id.$$

A result of Troyanov \cite{troyanov} shows that $T_{pl}(S,V)$ is
homeomorphic to $\R^{-3 \chi(S-V)}$. Below, we will use a natural
collection of charts on $T_{pl}$ which makes it a real analytic
manifold. Suppose $\T$ is a triangulation of $(S, V)$ with set of
edges $E=E(\T)$. Let
$$ \R^{E(\T)}_{\Delta} =\{ x \in \R_{>0}^E | x(e_i)+x(e_j)
> x(e_k), \text{if there is a triangle $t$ in $\T$ with edges $e_i, e_j, e_k$}\}$$
be the convex
polytope in $\R^E$.  For each $x \in \R^{E(\T)}_{\Delta}$, one
constructs a PL metric $d_x$ on $(S, V)$ by replacing each
triangle $t$ of edges $e_i, e_j, e_k$ by a Euclidean triangle of
edge lengths $x(e_i), x(e_j), x(e_k)$ and gluing them by
isometries along the corresponding edges. This construction
produces an injective map
$$ \Phi_{\T}:
\R^{E(\T)}_{\Delta} \to T_{pl}(S,V)$$ sending $x$ to 
$[d_x]$. 
 The image $P(\T) :=\Phi_{\T}(\R_{\Delta}^{E(\T)})$ is the space
of all PL metrics $[d]$ on $(S, V)$ for which $\T$ is isotopic to
a geometric triangulation in $d$. We call $x$  the \it length
coordinate \rm of $d_x$ and $[d_x]=\Phi_{\T}(x)$ with respect to
$\T$.    If $u: V \to \R$ is a discrete conformal factor and $x
\in \R_{>0}^E$, then the discrete conformal change $u*x$ of $x$ is
$u*x(vv') = x(vv') e^{ u(v)+u(v')}$ for all edges $vv' \in E(\T)$.
This is the prototype of (\ref{conformal}) introduced in \cite{rw}
and \cite{luo}.

In general $P(\T) \neq T_{pl}(S,V)$. Indeed, let $d$ be the metric
double of an obtuse triangle $t$ along its boundary and $\T$ be
the natural triangulation whose edges are edges of $t$. Let $\T'$
be the triangulation obtained by the diagonal switch at the
shortest edge of $t$. Then $\T'$ is not isotopic to any geometric
triangulation in $d$.

 Since each PL metric on $(S,V)$ admits a
geometric triangulation (for instance its Delaunay triangulation),
we see that $T_{pl}(S, V) =\cup_{\T} P(\T)$ where the union is
over all triangulations of $(S,V)$. The space $T_{pl}(S,V)$ is a
real analytic manifold with coordinate charts $\{ (P(\T),
\Phi_{\T}^{-1}) | \T \text{ triangulations of $(S,V)$}\}$. To see
transition functions $\Phi_{\T}^{-1}\Phi_{\T'}$ are real analytic,
note that  any two triangulations of $(S,V)$ are related by a
sequence of diagonal switches. Therefore, it suffices to show the
result for $\T$ and $\T'$ which are related  by a diagonal switch
along an edge $e$. In this case, the transition function
$\Phi_{\T}^{-1}\Phi_{\T'}$ sends $(x_0,x_1, ...., x_m)$ to
$(f(x_0, ..., x_m), x_1, ..., x_m)$ where $x_0$ is the length of
$e$ and $f$ is the length of the diagonal switched edge. See
figure~\ref{figure_2}. Let $t, t'$ be the triangles adjacent to $e$ so that the
lengths of edges of $t, t'$ are \{$x_0, x_1, x_2$\} and \{$x_0,
x_3, x_4$\}. Using the cosine law, we see that $f$ is a real
analytic function of $x_0, ..., x_4$. In the case that the
quadrilateral $t \cup_e t'$ is inscribed to a circle, we have the
famous Ptolemy identity
$x_0f =x_1x_3+x_2x_4.$

\begin{figure}[ht!]
\centering
\includegraphics[scale=0.70]{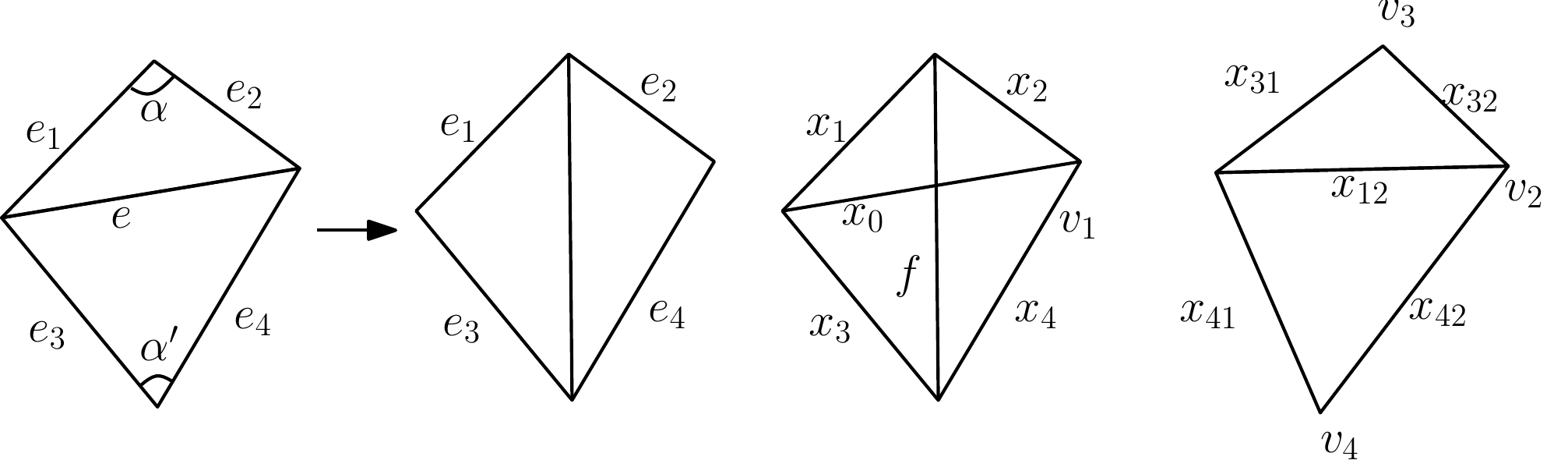}
\caption{diagonal switch and lengths of quadrilaterals}
\label{figure_2}
\end{figure}


\subsection{Delaunay triangulations}
Given a PL metric $d$ on $(S, V)$, its \it Voronoi decomposition
\rm is the collection of 2-cells $\{ R(v) | v \in V\}$ where
$R(v)=\{ x \in S | d(x, v) \leq d(x, v')$ for all $v' \in V\}$.
Its dual is called a \it Delaunay tessellation \rm $\mathcal C(d)$
of $(S,V, d)$ (\cite{aurenhamer}, \cite{bobenkosp}). It is a cell
decomposition
 of $(S,V, d)$ with vertices $V$ and  two vertices $v, v'$ are jointed by an edge if
  and only if $R(v) \cap R(v')$ is
1-dimensional. A \it Delaunay triangulation $\T(d)$ \rm of $(S,V)$
in metric $d$ is a geometric triangulation of the Delaunay
tessellation $\mathcal C(d)$ by further triangulating all
non-triangular 2-dimensional cells (without introducing extra
vertices). For a generic PL metric $d$, $\mathcal C(d)$ is a
Delaunay triangulation of $d$. 

\begin{lemma} \label{diagonalswitch} (See \cite{bobenkosp}, \cite{aurenhamer})
Each PL metric $d$ on $(S,V)$ has
a Delaunay triangulation.  If $\T$ and $\T'$ are Delaunay
triangulations of $d$, then there exists a sequence of Delaunay
triangulations $\T_1=\T$, $\T_2$, ..., $\T_k=\T'$ of $d$ so that
$T_{i+1}$ is obtained from $\T_i$ by a diagonal switch.
\end{lemma}

\begin{definition} (Delaunay cell) For a triangulation $\T$ of $(S, V)$, the
associated
Delaunay cell in $T_{pl}(S,V)$ is defined by
$$D_{pl}(\T)  =\{ [d] \in T_{pl}(S, V) |  \text{ $\T$ is isotopic to
a Delaunay triangulation of $d$} \}.$$
\end{definition}

Note that $D_{pl}(\T) \subset P(\T)$ and is non-empty. Indeed the
PL metric so that the length of each edge is 1  is in
$D_{pl}(\T)$. Assume that $\T$ is geometric in $d$.  One can
characterize PL metrics $[d] \in D_{pl}(\T)$ in terms
 of the length coordinate $x =\Phi_{\T}^{-1}([d])$ as follows.  By definition
 $\T$ is Delaunay in $d$ if and only if
 \begin{equation}\label{del}
 \alpha + \alpha' \leq \pi, \quad \text{i.e.,} \quad
\cos(\alpha)+\cos(\alpha') \geq 0, \quad \text{for each edge $e\in
E(\T)$}
\end{equation} where $\alpha, \alpha'$ are the two angles facing
$e$. See figure~\ref{figure_2}.
 Let $t$ and $t'$ be the triangles adjacent to $e$ and $e, e_1,
e_2$ be edges of $t$ and $e, e_3, e_4$ be the edge of $t'$. Note
that $t'=t$ is allowed. Suppose the length of $e$ (in $d$) is
$x_0$ and the length of $e_i$ is $x_i$, $i=1,...,4$. By the cosine
law, Delaunay condition (\ref{del}) is the same as

\begin{equation} \label{2.44}
 \frac{x_1^2+x_2^2-x_0^2}{2 x_1 x_2} + \frac{x_3^2 + x_4^2
-x_0^2}{2 x_3 x_4} \geq 0, \quad \text{for all edges $e \in
E(\T$)}.
\end{equation}

Inequality (\ref{2.44}) shows that $D_{pl}(\T) \subset T_{pl}$ is
bounded by a finite set of real analytic subvarieties. It turns
out $\{D_{pl}(\T)| \T\}$ forms a real analytic cell decomposition
of $T_{pl}$.

Let us recall the basics of real analytic cell decompositions of a
real analytic manifold $M^n$. A subspace $C \subset M$ is a \it
real analytic cell \rm if there is a real analytic diffeomorphism
$h$ defined in an open neighborhood $U$ of $C$ into $\R^n$ so that
$h(C)$ is  a  convex polytope in $\R^n$. A \it face \rm $C'$ of
$C$ is a subset so that $h(C')$ is a face of the polytope $h(C)$.
A \it real analytic cell decomposition \rm of $M$ is a locally
finite collection of n-dimensional real analytic cells $\{C_i | i
\in J\}$ so that $M =\cup_{i \in J} C_i$  and $C_{i_1} \cap ...
\cap C_{i_k}$ is a face of $C_{i_j}$ for all choices of indices.

A theorem of Rivin \cite{rivin} shows that $D_{pl}(\T)$ is a real
analytic cell of dimension $-3\chi(S-V)$. Indeed, one takes the
open neighborhood of $D_{pl}(\T)$ to be $P(\T)$ and fixes $e_1 \in
E$. Define  $h$ to be the real analytic map sending $x$ to
$(\phi_0(x), x(e_1))$ where $\phi_0(x) (e) =\alpha + \alpha'$
where $\alpha$ and $\alpha'$ are angles facing $e$. Rivin proved
that $h$ is a real analytic diffeomorphism into an open subset of
a codimension-1 affine subspace of $\R^E \times \R$ so that
$h(D_{pl}(\T))$ is a convex polytope and faces of $D_{pl}(\T)$ are
subsets defined by $\alpha+\alpha'=\pi$ for some collection of
edges $e$.
By \cite{aurenhamer}, \cite{bobenkosp}, if $W= D_{pl}(\T_1)
\cap.... \D_{pl}(\T_k) \neq \emptyset$, then $W$ is a face of
$D_{pl}(\T_i)$ for each $i$. Indeed, $W$ is the face of
$D_{pl}(\T_i)$ defined by the set of equalities:
$\alpha+\alpha'=\pi$ for all edges $e \notin \cap_{j=1}^k
E(\T_j)$.


The discussion above shows that  we have a real analytic cell
decomposition of the Teichm\"uller space by $\{D_{pl}(\T)| \T\}$
invariant under the action of the mapping class group,
\begin{equation}\label{2.1}
T_{pl}(S,V) =\cup_{[\T]} D_{pl}(\T)
\end{equation}
where the union is over  all isotopy classes $[\T]$ of
triangulations of $(S,V)$.


\section{Penner's work on decorated Teichm\"uller spaces}

One of the main tools used in our proof is the decorated
Teichm\"uller space theory developed by R. Penner \cite{penner}.
We will recall  the theory and prove a few new results in this
section. For details, see \cite{penner} or \cite{guoluo}.

\subsection{Decorated triangles}
Let $\H^2$ be the 2-dimensional hyperbolic plane. An \it ideal
triangle \rm  is a hyperbolic triangle in $\H^2$ with three
vertices $v_1, v_2, v_3$ at the sphere at infinity of $\H^2$. Any
two ideal triangles are isometric. A \it decorated ideal triangle
\rm $\tau$ is an ideal triangle so that each vertex $v_i$ is
assigned a horoball $H_i$ centered at $v_i$. Let $e_i$ be the
complete geodesic edge of $\tau$ opposite to the vertex $v_i$. The
inner \it angle \rm $a_i$ of $\tau$  is the length of the portion
of the horocycle $\partial H_i$ between  $e_j$ and $e_k$,
$\{i,j,k\}=\{1,2,3\}$.   The \it length \rm $l_i \in \R$ of the
edge $e_i$ in $\tau$ is the signed distance between $H_j$ and
$H_k$ ($j,k \neq i$). To be more precise, if $H_j \cap H_k
=\emptyset$, then $l_i>0$ is the distance between $H_k$ and $H_j$.
If $H_j \cap H_k \neq \emptyset$, then $-l_i$ is the distance
between two end points of $\partial (e_i \cap H_j \cap H_k)$.
Penner calls $L_i =e^{l_i/2}$ the \it $\lambda$-length \rm of
$e_i$.

\begin{figure}[ht!]
\centering
\includegraphics[scale=0.8]{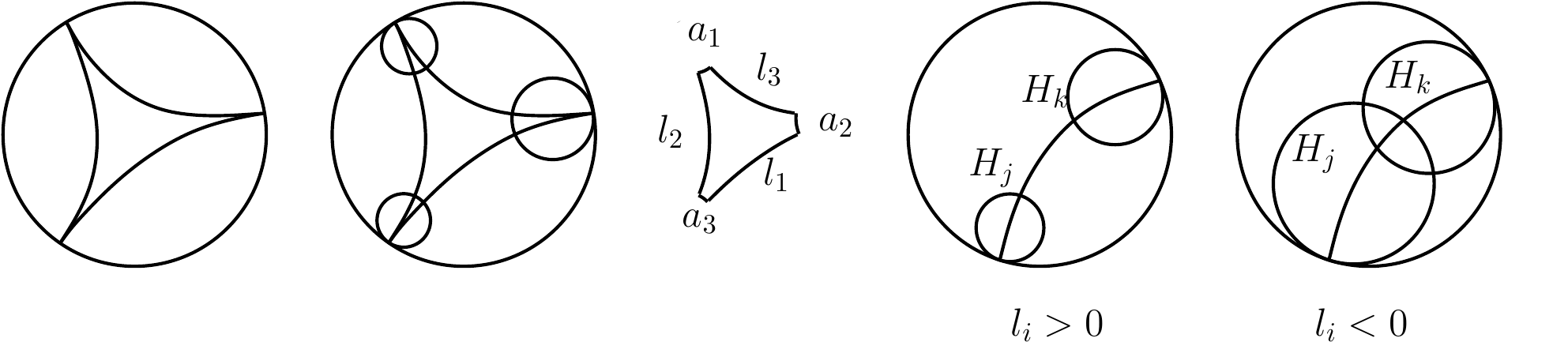}
\caption{decorated ideal triangles and their edge lengths}
\label{figure_3}
\end{figure}

 It is known that for any $l_1, l_2, l_3 \in
\R$, there exists a unique decorated ideal triangle of edge
lengths $l_1, l_2, l_3$. The relationship between the lengths
$l_i$ and angles $a_j$'s is the following \it cosine law \rm
proved by Penner:
\begin{equation}\label{1}
a_i=e^{\frac{1}{2}(l_i -l_j -l_k)}=\frac{L_i}{L_jL_k},  \quad
\quad \ln(a_i)+\ln(a_j)=-l_k,  \quad \quad \{i,j,k\}=\{1,2,3\}.
\end{equation}

Let $S$ be a closed connected surface and $V=\{v_1, ..., v_n\}
\subset S$  and $\Sigma =S-V$. We assume $n \geq 1$ and
$\chi(\Sigma) <0$.  Following Penner, a \it decorated hyperbolic
metric \rm on $\Sigma$ is a complete finite area hyperbolic metric
$d$ on $\Sigma$ together with a horoball $H_i$ centered at the
i-th cusp at $v_i$ for each $i$. We can also parameterize it as
$(d, w)$ where $w=(w_1, ..., w_n) \in \R^n_{>0}$ with $w_i$ being
the length of the horocycle $\partial H_i$. Two decorated
hyperbolic metrics on $\Sigma$ are \it equivalent \rm if there is
an isometry $h$ between them so that $h$ is homotopic to the
identity and $h$ preserves the horoballs. The space of all
equivalence classes of decorated hyperbolic metrics on $\Sigma$ is
defined to be the \it decorated Teichm\"uller space \rm
$T_D(\Sigma)$. If we use $T(\Sigma)$ to denote the usual
Teichm\"uller space of complete hyperbolic metrics of finite area
on $\Sigma$,
 then there is a natural homeomorphism from $T_D(\Sigma)$ to $T(\Sigma)
\times \R^n_{>0}$ by sending $[(d, w)]$ to $([d], w)$. The
projection $T_D(\Sigma) \to T(\Sigma)$ sending $[(d, w)]$ to $[d]$
records the underlying hyperbolic metric.

Now suppose $\T$ is a triangulation of $(S,V)$ with $E=E(\T)$. %
 Then Penner
introduced a homeomorphism map $\Psi_{\T}: \R_{>0}^E \to
T_D(\Sigma)$ called $\lambda$-length coordinate as follows.
 For each $x \in \R_{>0}^E$, i.e.,
$x: E \to \R_{>0}$, $\Psi_{\T}(x)$ is the equivalence class of the
decorated hyperbolic metric $(d, w)$ on $\Sigma$ obtained as
follows. If $t$ is a triangle in $\T$ with three edges $e_i, e_j,
e_k$, one replaces $t$ by the decorated ideal triangle of edge
lengths $2\ln x(e_i), 2 \ln x(e_j)$ and $2 \ln x(e_k)$ and glues
these decorated ideal triangles isometrically along the
corresponding edges preserving decoration. One obtains a decorated
hyperbolic metric $(d,w)$ on $\Sigma$. The horoballs are the
gluing of the corresponding portions of horoballs associated to
ideal triangles. In particular, $w_i$ is the sum of all angles of
the decorated ideal triangles at $v_i$. Penner proved, using his
Ptolemy identity, that $\Psi_{\T}^{-1} \Psi_{\T'}$ is real
analytic for any two triangulations $\T$ and $\T'$. Here Ptolemy
identity for decorated ideal quadrilaterals states that
$AA'+BB'=CC'$ where $A,A',B,B'$ are the $\lambda$-lengths of the
edges of a quadrilateral and $C,C'$ are the $\lambda$-lengths of
the diagonals. See figure~\ref{figure_4}.  In particular,
$\{\Psi_{\T}|\T\}$ forms real analytic charts for $T_D(\Sigma)$.

The following lemma is well know. We omit the proof.
\begin{lemma}\label{l1}  Suppose $C$ is an embedded  horocycle of length $w_i$
 centered at a cusp in a
complete hyperbolic surface and $C'$ is another embedded horocycle
of smaller length $w_i'$ centered at the same cusp. Then the $w_i
=w_i' e^t$ where $t=d(C, C')$ is the distance between $C$ and
$C'$.
\end{lemma}
By the lemma and definition, if $\Psi_{\T}(x)=[(d,w)]$ then for
any $k>0$, $\Psi_{\T}(kx) =[(d,\frac{1}{k}w)]$. Thus, for any $(d,
w)$, by choosing $k$ large, one may assume the associated
horoballs are disjoint and embedded in $(d, w/k)$.

\subsection{Delaunay triangulations}
Given a decorated hyperbolic metric $(d, w)$ on $\Sigma$, there is
a natural \it Delaunay triangulation \rm $\T$ associated to $(d,
w)$. The geometric definition of $\T$ goes as follows. First
assume that the associated horoballs $H_1(w), ..., H_n(w)$ are
embedded and disjoint in $\Sigma$. Consider the Voronoi cell
decomposition of the compact surface $X_w=\Sigma-\cup_{i=1}^n
int(H_i(w))$ so that the 2-cell $R_i(w)$ associated to $v_i$ is
$\{ x \in X_w| d(x,
\partial H_i(w)) \leq d(x,\partial H_j(w)), \text{ all $j$}\}$. An
\it orthogeodesic \rm in $X_w$ is a geodesic from $\partial X_w$
to $\partial X_w$ perpendicular to $\partial X_w$. The dual of the
Voronoi decomposition is a decomposition $\mathcal C(d,w)$ of $X$
by a collection of disjoint embedded orthogeodesics arcs $\{s'\}$
constructed as follows. If $s \subset R_i(w) \cap R_j(w)$ is a
geodesic segment,  take a point $p \in S$ and consider the two
shortest geodesics $b_i$ and $b_j$ in $R_i(w)$ and $R_j(w)$
respectively from $p$ to $\partial H_i(w)$ and $\partial H_j(w)$.
The shortest orthogeodesic $s'$ in $X_w$ homotopic to $b^{-1}_i*
b_j$ is an arc in $\mathcal D(d)$ dual to $s$.
 A \it
Delaunay triangulation of $X_w$ \rm is a further decomposition of
$\mathcal C(d,w)$ by decomposing all non-hexagonal 2-cells by
orthogeodesic.  Since each orthogeodesic extends to a complete
geodesic from cusp to cusp, one obtains a \it Delaunay
triangulation $\T(d,w)$ \rm of the decorated metric $(d, w)$ on
$\Sigma$ by extension.   For a generic metric $(d, w)$, a Delaunay
triangulation is the dual to the Voronoi decomposition.
 By the
definition of Voronoi cells and lemma \ref{l1}, Delaunay
triangulations  of $(d, w)$ and $(d, w/k)$ are the same when
$k>1$. Due to this, for a general decorated metric $(d, w)$, we
define a Delaunay triangulation of $(d, w)$ to be that of $(d,
w/k)$ for $k$ large.



For a given triangulation $\T$ of $(S,V)$, let $D(\T)$ be the set
of all equivalence classes of decorated hyperbolic metrics $(d,w)$
in $T_D(\Sigma)$ so that  $\T$ is isotopic to a Delaunay
triangulation of $(d,w)$. Penner proved the following important
theorem in \cite{penner}. Details on the real analytic
diffeomorphism part of the decomposition can be found in
\cite{guoluo}.
\begin{theorem} (Penner)  The decorated Teichm\"uller space $T_D(\Sigma)$
has a real analytic cell decomposition by $\{D(\T) | \T\}$ and
$$ T_D(\Sigma) =\cup_{[\T]} D(\T)$$ where
the union is over all isotopy classes
of triangulations. The decomposition is invariant under the action
of the mapping class group.
\end{theorem}

\subsection{Finite set of Delaunay triangulations}

We thank B. Springborn for informing us the following result was
known before and was a theorem of Akiyoshi \cite{ak}. However, our
proof is different and short. For completeness, we present our
proof in the appendix.  The theorem holds for decorated finite
volume hyperbolic manifolds of any dimension.

\begin{theorem}(Akiyoshi)\label{dau} For any finite area complete hyperbolic metric $d$
on $\Sigma$, there are only finitely many isotopy classes of
triangulations $\T$ so that $([d] \times \R_{>0}^n) \cap D(\T)
\neq \emptyset$. In particular, there exist triangulations $\T_1$,
..., $\T_k$ so that
 for any $w \in R_{>0}^n$, any Delaunay
triangulation  $(d, w)$ is isotopic to one of
$\T_i$. 
\end{theorem}

\section{Euclidean polyhedral metrics and decorated hyperbolic
metrics}

The relationship between edge length coordinate of PL metrics with
that of $\lambda$-length was first noticed in \cite{bps}. Fix a
triangulation $\T$ of $(S,V)$, we have two coordinate maps
$\Phi^{-1}_{\T}: P(\T) \to \R^{E(\T)}$ and  $\Psi_{\T}: \R^{E(\T)}
\to T_D(S,V)$. Consider the injective map $A_{\T}: P(\T) \to
T_D(\Sigma)$ defined by $\Psi_{\T} \circ \Phi_{\T}^{-1}$.

\begin{theorem} $A_{\T}|_{D_{pl}(\T)}$ is a real analytic diffeomorphism  from $D_{pl}(\T)$ onto
$D(\T)$.
\end{theorem}
\begin{proof}  To see that $A_{\T}$ maps $D_{pl}(\T)$ bijectively onto
$D(\T)$, it suffices to
 show that $\Phi_{\T}^{-1}(D_{pl}(\T)) =
\Psi_{\T}^{-1}(D(\T))$.

Recall that the characterization of a PL metric $d$ which is
Delaunay in $\T$ in terms of  $x =\Phi_{\T}^{-1}(d)$ is as
follows. Take an edge $e \in E(\T)$ and let $t$ and $t'$ be the
triangles adjacent to $e$ so that $e, e_1, e_2$ are edges of $t$
and $e, e_3, e_4$ are the edge of $t'$. Suppose $\alpha, \alpha'$
are the angles (measured in $d$) in $t$ and $t'$ facing $e$. Then
the Delaunay condition is equivalent to
\begin{equation}\label{delaunay}
\alpha + \alpha' \leq \pi, \quad \text{i.e.,} \quad
\cos(\alpha)+\cos(\alpha') \geq 0, \quad \text{for all edges $e
\in E(\T)$}.
\end{equation}
 Suppose the
length of $e$ (in $d$) is $x_0$ and the length of $e_i$ is $x_i$,
$i=1,...,4$. By the cosine law, Delaunay condition
(\ref{delaunay}) is the same as

\begin{equation} \label{2.46}
 \frac{x_1^2+x_2^2-x_0^2}{2 x_1 x_2} + \frac{x_3^2 + x_4^2
-x_0^2}{2 x_3 x_4} \geq 0, \quad \text{for all edges $e \in
E(\T)$}.
\end{equation}

This shows that $$\Phi_{\T}^{-1}(D_{pl}(\T))=\{ x \in \R_{>0}^E |
\text{ (\ref{2.46}) holds for each edge $e$, and  (\ref{2.47})
holds for each triangle}\}$$ where
\begin{equation}\label{2.47}
x(e_i)+x(e_j)> x(e_k),  \quad \text{ $e_i, e_j, e_k$ form edges of
a triangle in $\T$}. \end{equation}

\begin{lemma}\label{deltri}  Suppose $x: E(\T) \to \R_{>0}$ so
that (\ref{2.46}) holds for all edges. Then (\ref{2.47}) holds for
all triangles.
\end{lemma}
\begin{proof}
Suppose otherwise, there exists  $x \in \R_{>0}^E$ so that
(\ref{2.46}) holds but there is a triangle with edges $e_i, e_j,
e_k$ so that \begin{equation}\label{15} x(e_i) \geq x(e_j) +
x(e_k).
\end{equation} In this case, we say $e_i$ is a "bad" edge. Let $e$
be a "bad" edge of the largest $x$ value, i.e., $x(e) =\max\{
x(e_i) |$ (\ref{15}) holds\}. Let $t$, $t'$ be the triangles
adjacent to $e$ and the edges of $t$ and $t'$ be $\{e, e_1, e_2\}$
and $\{e, e_3, e_4\}$. Note that $t'=t$ is allowed if $e$ is
adjacent to one triangle. Let $x_0 =x(e)$, $x_i =x(e_i)$ for
$i=1,2,3,4$. Without loss of generality we may assume that
\begin{equation} \label{3.1}
 x_1+x_2 \leq x_0.
\end{equation}
Since $e$ is a "bad" edge of the largest $x$ value, we have $x_3 <
x_0+x_4$ and $x_4 < x_0 +x_3$, i.e.,
\begin{equation}\label{3.2}
|x_3 -x_4|< x_0. \end{equation}

On the other hand, inequality (\ref{2.46}) holds for $x_0, x_1,
..., x_4$, i.e.,
\begin{equation}\label{3.6}
\frac{x_0^2 -x_1^2-x_2^2}{2 x_1 x_2} \leq \frac{x_3^2 +x_4^2
-x_0^2}{2 x_3 x_4}. \end{equation}

Inequality (\ref{3.1}) says the left-hand-side of (\ref{3.6}) is
at least 1 and inequality (\ref{3.2}) says the right-hand-side of
(\ref{3.6}) is strictly less than 1. This is a contradiction.
$\square$
\end{proof}

The space $ \Psi_{\T}^{-1}(D(\T))$ can be characterized as
follows. Suppose the $\lambda$-length of $(d',w)
\in D(\T)$ is $x=\Psi_{\T}^{-1}(d',w)$. 
For each edge $e$ in $(S, \T, d')$, let $a,a'$ be the two angles
facing $e$ and $b,b', c, c'$ be the angles adjacent to the edge
$e$.  Then $(d',w)$ is Delaunay in $\T$ if and only if for each
edge $e \in E(\T)$ (see \cite{penner} or \cite{guoluo}),
\begin{equation} \label{2}
 a+a' \leq b+b'+c+c'.
\end{equation}

Let $t$ and $t'$ be the triangle adjacent to $e$ and $e, e_1, e_2$
be edges of $t$ and $e, e_3, e_4$ be the edges of $t'$. Let the
$\lambda$-length of $e$ be $x_0$ and the $\lambda$-length of $e_i$
be $x_i$. Then using the cosine law (\ref{1}), one sees that
(\ref{2}) is equivalent to
\begin{equation} \label{4}
\frac{x_0^2}{x_1 x_2}+\frac{x_0^2}{x_3 x_4}
 \leq
\frac{x_1}{x_2}+\frac{x_2}{x_1}+\frac{x_3}{x_4}+\frac{x_4}{x_3},
\quad  \text{for each $e \in E(\T)$}.
\end{equation}

Inequality (\ref{4}) is equivalent to
\begin{equation} \label{3.5}
0 \leq  \frac{x_1^2+x_2^2-x_0^2}{2 x_1 x_2} + \frac{x_3^2 + x_4^2
-x_0^2}{2 x_3 x_4}, \quad \text{for each $e \in E(\T)$}.
\end{equation}

Therefore, $$ \Psi_{\T}^{-1}(D(\T))=\{ x \in \R_{>0}^E | \text{
(\ref{3.5}) holds at each edge $e\in E(\T)$}\}.$$

However, inequality (\ref{3.5}) is the same as (\ref{2.46}). This
shows $\Phi_{\T}^{-1}(D_{pl}(\T)) \subset \Psi_{\T}^{-1}(D(\T))$.
On the other hand, lemma \ref{deltri} implies that
$\Phi_{\T}^{-1}(D_{pl}(\T)) = \Psi_{\T}^{-1}(D(\T))$.

Finally, since both $\Phi_{\T}$ and $\Psi_{\T}$ are real analytic
diffeomorphisms and $A_{\T}=\Psi_{\T} \circ \Phi_{\T}^{-1}$ and
$A^{-1}_{\T} =\Phi_{\T} \circ \Psi^{-1}_{\T}$, we see that
$A_{\T}$ is a real analytic diffeomorphism. $\square$
\end{proof}

\subsection{Globally defined map, diagonal switch and Ptolemy relation}

\begin{theorem} Suppose $\T$ and $\T'$ are two triangulations of
$(S,V)$ so that $D_{pl}(\T) \cap D_{pl}(\T') \neq \emptyset$. Then
\begin{equation}\label{343} A_{\T}|_{D_{pl}(\T) \cap
D_{pl}(\T')} = A_{\T'}|_{D_{pl}(\T) \cap D_{pl}(\T')}.
\end{equation}
 In particular, the gluing of these $\A_{\T}|_{D_{pl}(\T)}$
mappings produces a homeomorphism $\A=\cup_{\T}
\A_{\T}|_{D_{pl}(\T)}: T_{pl}(S,V) \to T_D(S-V)$ such that
$A([d])$ and $A([d'])$ have the same underlying hyperbolic
structure if and only if $d$ and $d'$ are discrete conformal.
\end{theorem}

\begin{proof} Suppose $[d] \in D_{pl}(\T) \cap D_{pl}(\T')$, i.e.,
$\T$ and $\T'$ are both Delaunay in the PL metric $d$. Then it is
known that there exists a sequence of triangulations $\T_1 =\T,
\T_2, ..., T_k =\T'$ on $(S,V)$ so that each $\T_i$ is Delaunay in
$d$ and $\T_{i+1}$ is obtained from $\T_i$ by a diagonal switch.
In particular, $A_{\T}([d]) =A_{\T'}([d])$ follows from
$A_{\T_i}([d])=A_{\T_{i+1}}([d])$ for $i=1,2,..., k-1$.
 Thus, it suffices to show $A_{\T}([d]) =A_{\T'}([d])$ when $\T'$ is obtained from
$\T$ by a diagonal switch along an edge $e$. In this case the
transition functions $\Phi_{\T}^{-1}\Phi_{\T'}$ and
$\Psi_{\T}^{-1} \Psi_{\T'}$ are the diagonal switch formulas.
Penner proved an amazing result  that the $\lambda$-lengths
satisfy the Ptolemy identity for decorated
ideal quadrilaterals. See \cite{penner} and figure~\ref{figure_4}. 
  This
result, translated into the language of length coordinates, says
that $\Phi_{\T}^{-1}\Phi_{\T'}(x) =\Psi_{\T}^{-1} \Psi_{\T'}(x)$
for $x \in \Phi^{-1}_{\T} (D_{pl}(\T) \cap D_{pl}(\T'))$. This is
the same as (\ref{343}).  Taking the inverse, we obtain
\begin{equation}\label{344} A_{\T}^{-1}|_{D(\T) \cap D(\T')} =
A_{\T'}^{-1}|_{D(\T) \cap D(\T')}. \end{equation}

\begin{lemma} (a) $D_{pl}(\T) \cap
D_{pl}(\T') \neq \emptyset$ if and only if $D(\T) \cap D(\T') \neq
\emptyset$.

(b) The gluing map $\A =\cup_{\T} \A_{\T}|_{D_{pl}(\T)}: T_{pl}
\to T_D$ is a homeomorphism invariant under the action of the
mapping class group.
\end{lemma}
\begin{proof} By (\ref{343}) and (\ref{344}), the maps  $\A =\cup_{\T} \A_{\T}|_{D_{pl}(\T)}: T_{pl} \to T_D$
and $\B =\cup_{\T} \A^{-1}_{\T}|_{D(\T)}: T_D \to T_{pl}$ are well
defined and  continuous. Since $\A(D_{pl}(\T) \cap D_{pl}(\T') )
\subset D(\T) \cap D(\T')$ and $\B(D(\T) \cap D(\T')) \subset
D_{pl}(\T) \cap D_{pl}(\T')$, part (a) follows. To see part (b),
since $T_D =\cup_{\T} D(\T)$, the map $\A$ is onto. To see $\A$ is
injective, suppose $x_1 \in D_{pl}(\T_1), x_2 \in D_{pl}(\T_2)$ so
that $\A(x_1) =\A(x_2) \in D(\T_1) \cap D(\T_2)$. Apply
(\ref{344}) to $\A^{-1}_{\T_1}|, \A^{-1}_{\T_2}|$ on the set
$D(\T_1) \cap D(\T_2)$ at the point $A(x_1)$, we conclude that
$x_1=x_2$. This shows that $\A$ is a bijection with inverse $\B$.
Since both $\A$ and $\B$ are continuous, $\A$ is a homeomorphism.
$\square$
\end{proof}

 Now if $d$ and $d'$ are two discrete
conformally equivalent PL metrics, then  $\A([d])$ and $\A([d'])$
are of the form $(p, w)$ and $(p, w')$ due to the definition of
$\Psi^{-1}_{\T} \Phi_{\T}$. On the other hand, if two PL metrics
$d, d'$ satisfy that $\A([d])$ and $\A([d'])$ are of the form $(p,
w)$ and $(p, w')$, consider a generic smooth  path $\gamma(t)=(p,
w(t)), t \in [0,1]$, in $T_D(\Sigma)$ from $(p,w)$ to $(p, w')$ so
that $\gamma(t)$ intersects the cells $D(\T)$'s transversely. This
implies that $\gamma$ passes through a finite set of cells
$D(\T_i)$ and $\T_j$ and $\T_{j+1}$ are related by a diagonal
switch. Let
 $t_0=0<... <t_m=1$ be a partition of $[0,1]$ so that $\gamma([t_i,
t_{i+1}]) \subset D(\T_i)$. Say $d_i$ is the PL metric so that
$\A([d_i]) =\gamma(t_i) \in D(\T_i) \cap D(\T_{i+1})$, $d_1=d$ and
$d_m=d'$. Then by definition, the sequences $\{d_1, ..., d_m\}$
and the associated Delaunay triangulations $\{\T_1, ..., T_m\}$
satisfy the definition of discrete conformality for $d,d'$.
$\square$
\end{proof}

\begin{theorem} The homeomorphism $\A: T_{pl}(S,V) \to T_D(S-V)$
is a $C^1$ diffeomorphism.
\end{theorem}

\begin{proof} It suffices to show that for a point $[d] \in D_{pl}(\T)
\cap \D_{pl}(\T')$, the derivatives $DA_{\T}[d])$ and
$DA_{\T'}([d])$ are the same. Since  both $\T$ and $\T'$ are
Delaunay in $d$ and are related by a sequence of Delaunay
triangulations (in $d$) $\T_1=\T, \T_2, ..., \T_k =\T'$,
$DA_{\T}([d])=DA_{\T'}([d])$ follows from $DA_{\T_i}([d])
=DA_{\T_{i+1}}([d])$ for $i=1,2,..., k-1$. Therefore, it suffices
to show $DA_{\T}([d]) = DA_{\T'}([d])$ when $\T$ and $\T'$ are
related by a diagonal switch at an edge $e$. In the coordinates
$\Phi_{\T}$ and $\Psi_{\T}$, the
 fact that $DA_{\T}([d]) = DA_{\T'}([d])$  is equivalent to  the following smoothness
question on the diagonal lengths.

\begin{lemma} Suppose $Q$ is a convex Euclidean quadrilateral
whose four edges are of lengths $x,y,z,w$ and the length of a
diagonal is $a$. See figure~\ref{figure_4}. Suppose $A(x,y,z,w,a)$ is
the length of the other diagonal and
$B(x,y,z,w,a)=\frac{xz+yw}{a}$. If a point $(x,y,z,w,a)$ satisfies
$A(x,y,z,w,a)=B(x,y,z,w,a)$, i.e., $Q$ is inscribed in a circle,
then $DA(x,y,z,w,a)=DB(x,y,z,w,a)$ where $DA$ is the derivative of
$A$.
\end{lemma}

\begin{figure}[ht!]
\centering
\includegraphics[scale=0.6]{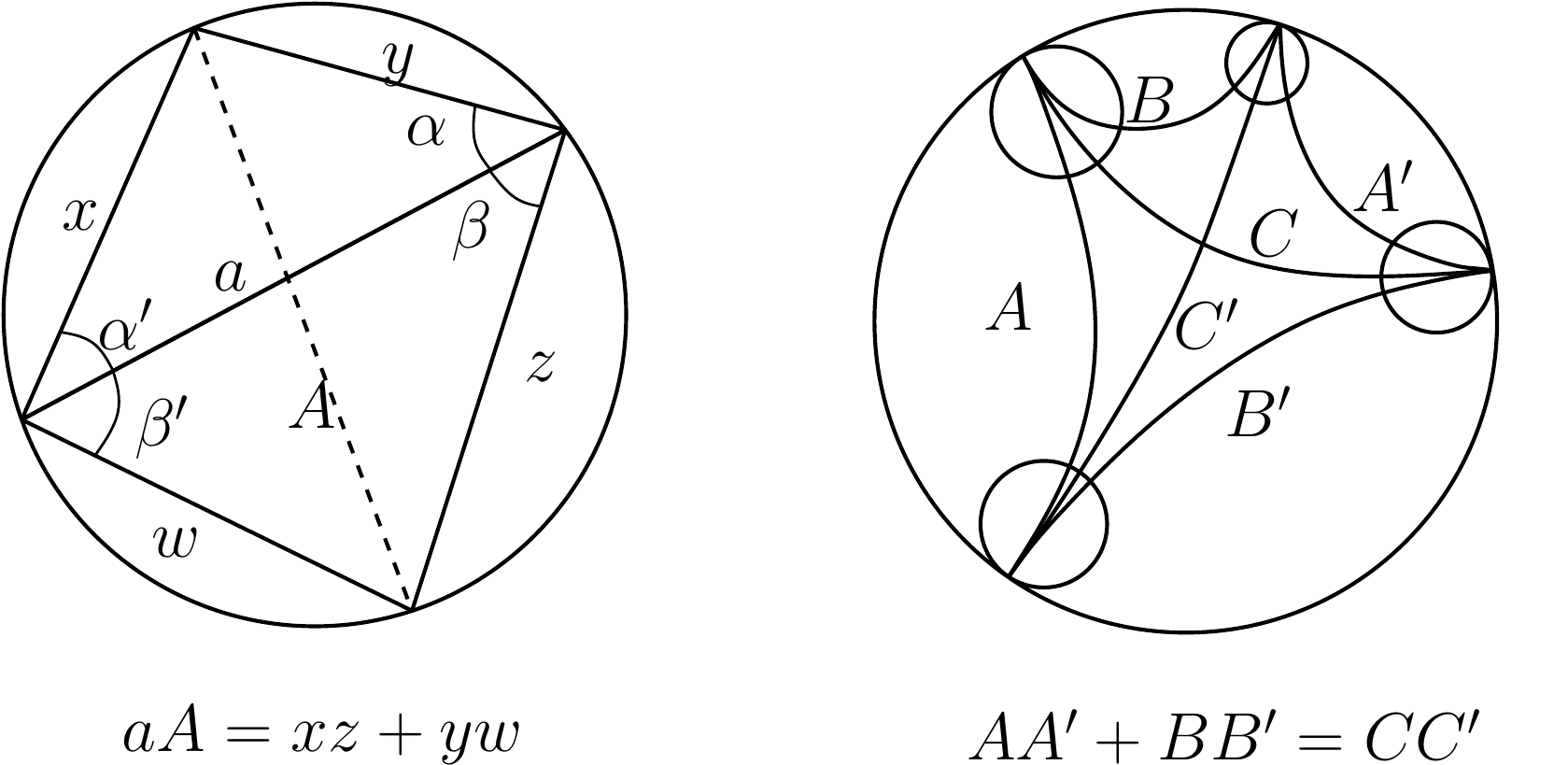}
\caption{Euclidean and hyperbolic Ptolemy}
\label{figure_4}
\end{figure}

\begin{proof}
The roles of $x,y,z,w$ are symmetric with respect to $a$. Hence it
suffices to show that $\frac{\partial A}{\partial
a}=\frac{\partial B}{\partial a}$ and $\frac{\partial A}{\partial
x}=\frac{\partial B}{\partial x}$ at these points. First, we have
$\frac{\partial B}{\partial x}=\frac{z}{a}$ and $\frac{\partial
B}{\partial a}= -\frac{B}{a}$.

Now let $\alpha, \alpha', \beta, \beta'$ be the angles formed by
the pairs of edges $\{y,a\}$, $\{a, x\}$, $\{a, z\}$ and
$\{a,w\}$.  By the cosine law, we have
$$ A^2 = y^2+z^2-2yz \cos(\alpha + \beta).$$
Take partial $x$ derivative of it. We obtain
$$ 2A \frac{\partial A}{\partial x} = 2 y z \sin(\alpha + \beta)
\frac{\partial \alpha}{\partial x}.$$ But it is well known (see
for instance \cite{luo1}) that in the triangle of lengths $x,y,a$,
\begin{equation}\label{123}
\frac{\partial \alpha}{\partial x}=\frac{x}{ay \sin(\alpha)}.
\end{equation}
Therefore, $$\frac{\partial A}{\partial
x}=\frac{xz\sin(\alpha+\beta)}{aA \sin(\alpha)}.$$ Now at the
point where $A(x,y,z,w,a)=B(x,y,z,w,a)$, the quadrilateral is
inscribed to the circle. Therefore,
$\frac{\sin(\alpha+\beta)}{\sin(\alpha)}=\frac{A}{x}$. By putting
these together, we see that $\frac{\partial A}{\partial
x}=\frac{xzA}{aAx} =\frac{z}{a} =\frac{\partial B}{\partial x}$.

Next, we calculate $\frac{\partial A}{\partial a}$. By the formula
above, we obtain $2A \frac{\partial A}{\partial a} = 2 yz
\sin(\alpha+\beta)(\frac{\partial \alpha}{\partial
a}+\frac{\partial \beta}{\partial a}).$ Now by the derivative
cosine law (\cite{chowluo}), we have $\frac{\partial
\alpha}{\partial a}=-\frac{\partial \alpha}{\partial
x}\cos(\alpha')$ which in turn is  $-\frac{x \cos(\alpha')}{a y
\sin(\alpha)}$ by (\ref{123}). Similarly, we have $\frac{\partial
\beta}{\partial a} = -\frac{w \cos(\beta')}{az \sin(\beta)}$.
Putting these together, we obtain,
$$ \frac{\partial A}{\partial a} = -\frac{yz \sin(\alpha+\beta)}{aA}
(\frac{x \cos(\alpha')}{y \sin(\alpha)} +\frac{w \cos(\beta')}{z
\sin(\beta)}).$$  Now since $A=B$, the quadrilateral is inscribed
in a circle, therefore, $\frac{\sin(\alpha+\beta)}{\sin(\alpha)}
=\frac{A}{x}$ and
$\frac{\sin(\alpha+\beta)}{\sin(\beta)}=\frac{A}{w}.$ Therefore,
$\frac{\partial A}{\partial a}=- \frac{1}{a}( z \cos(\alpha')+y
\cos(\beta')) =-\frac{A}{a}=-\frac{B}{a}=\frac{\partial
B}{\partial a}$ where the identity $A=z \cos(\alpha')+y
\cos(\beta')$ comes from the triangle of lengths $y,z,A$ and the
fact that $Q$ is inscribed in a circle. $\square$

\end{proof}

\end{proof}
\section{A proof of the main theorem}

Using the map $\A: T_{pl}(S,V) \to T_D(\Sigma)$, we see that for a
given PL metric $d$ on $(S,V)$, the set $\{[d']|\text{$d'$ is
discrete}$   $ \text{ conformal to $d$}\}$ is $C^1$-diffeomorphic
to $\{p\} \times \R_{>0}^n \subset T_D(S-V)$ for some $p \in
T(\Sigma)$. Therefore, the discrete uniformization theorem is
equivalent to a statement about the discrete curvature map defined
on $\{p\} \times \R_{>0}^n \subset
 T_D(S-V)$.  Let us make a change of variables from $w=(w_1, ...,
w_n) \in R_{>0}^n$ to $u=(u_1, ..., u_n) \in \R^n$ where $u_i
=\ln(w_i)$. We write $w=w(u)$. For a given $p \in T(\Sigma)$,
define the curvature map $\mathbf F: \R^n \to (-\infty, 2\pi)^n$
by
\begin{equation}\label{curv}\mathbf F(u) =K_{\A^{-1}(p, w(u))}
\end{equation} where $K_d$ is the discrete curvature.
 The map satisfies the property that $\mathbf
F(u+k(1,1,...,1))=\mathbf F(u)$ and $\mathbf F(u)$ lies in the
plane $GB=\{ x \in \R^n | \sum_{i=1}^n x_i =2\pi \chi(S)\}$
defined by the Gauss-Bonnet identity. Let $P=\{u \in \R^n |
\sum_{i=1}^n u_i=0\}$ and $Q =GB \cap (-\infty, 2\pi)^n$. Then the
restriction $F:=\mathbf F|_{P}: P \to Q$. The discrete
uniformization theorem is equivalent to say that $F: P \to Q$ is a
bijection. We will show that $F$ is a homeomorphism in this
section.

We will prove that $F: P \to Q$ is injective in \S 5.2 using a
variational principle developed in \cite{luo}.  Assuming
injectivity, we show that $F: P \to Q$ is onto in \S 5.1.

\subsection{The map $F$ is onto}

Assuming that $F$ is injective, we prove $F$ is onto in this
section. Since both $P$ and $Q$ are connected manifolds of
dimension $n-1$ and $F$ is injective and continuous, it follows
that $F(P)$ is open in $Q$. To show that $F$ is onto, it suffices
to prove that $F(P)$ is closed in $Q$.

To this end, take a sequence $\{u^{(m)}\}$ in $P$ which leaves
every compact set in $P$. We will show that $\{F(u^{(m)})\}$
leaves each compact set in $Q$. By taking subsequences, we may
assume that for each index $i=1,2,...,n$, the limit $\lim_{m}
u^{(m)}_i = t_i$ exists in $[-\infty, \infty]$. Furthermore, since
the space $\{p\} \times P$ is in the union of a finite set of
Delaunay cells $D(\T)$, we may  assume, after taking another
subsequence, that the corresponding PL metrics $d_m= \A^{-1}(p,
w(u^{(m)}))$ are Delaunay in one triangulation $\T$. We will do
our calculation in the length coordinate $\Phi_{\T}$ below.

Due to the normalization that $\sum_{i} u^{(m)}_i=0$ and $u^{(m)}$
does not converge to any vector in $P$, there exists $t_i=\infty$
and $t_j=-\infty$. Let us label vertices $v \in V$ by \it black
\rm and \it white \rm as follows. The vertex $v_i$ is black if and
only if $t_i=-\infty$ and all other vertices are white.

\begin{lemma} (a) There does not exist a triangle $\tau \in \T$ with
exactly two white vertices.

(b) If $\Delta v_1v_2v_3$  is a triangle in $\T$ with exactly one
white vertex at $v_1$, then the inner angle of the triangle at
$v_1$ converges to $0$ as $m \to \infty$ in the metrics $d_m$.

\end{lemma}

\begin{proof}
To see (a), suppose otherwise, using the $\Phi_{\T}$ length
coordinate, we see the given assumption is equivalent to
following. There exists a Euclidean triangle of lengths $a_i
e^{u^{(m)}_j +u^{(m)}_k}$, $\{i,j,k\}=\{1,2,3\}$, where  $\lim_{m}
u^{(m)}_i >- \infty$ for $i=2,3$ and $\lim_{m} u^{(m)}_1
=-\infty$. By the triangle inequality, we have $$a_2 e^{u^{(m)}_1
+u^{(m)}_3} + a_3 e^{u^{(m)}_1 +u^{(m)}_2} > a_1 e^{u^{(m)}_2
+u^{(m)}_3}$$ This is the same as
$$a_2 e^{-u^{(m)}_2} + a_3 e^{-u^{(m)}_3} > a_1 e^{-u^{(m)}_1}.$$
However, by the assumption, the right-hand-side tends to $\infty$
and the left-hand-side is bounded. The contradiction shows that
(a) holds.

To see (b), we use the same notation as in the proof of (a).
 Let the length $l^{(m)}_i$ of the edge
 $v_j v_k$ in metric $d_m$  be $a_i e^{u^{(m)}_j
+u^{(m)}_k}$, $\{i,j,k\}=\{1,2,3\}$. Let $\alpha_i: =\alpha_i(m)$
be the inner angle at $v_i$. Note that the triangle is similar to
the triangle of lengths $a_ie^{-u^{(m)}_i}$ and $\lim_{m}
a_ie^{-u^{(m)}_i}$ is $\infty$ when  $i=2,3$ and is finite for
$i=1$. Therefore, the angle $\alpha_1$ tends to 0. $\square$
\end{proof}

We now finish the proof of $F(P)=Q$ as follows. Since the surface
$S$ is connected, there exists an edge $e$ whose end points $v,
v_1$ have different colors. Assume $v$ is white and $v_1$ is
black. Let $v_1, ..., v_k$ be the set of all  vertices adjacent to
$v$ so that $v, v_i, v_{i+1}$ form vertices of a triangle and let
$v_{k+1}=v_1$. Now apply above lemma to triangle $\Delta vv_1v_2$
with $v$ white and $v_1$ black, we conclude that $v_2$ must be
black. Repeating this to $\Delta vv_2v_3$ with $v$  white and
$v_2$ black, we conclude $v_3$ is black. Inductively, we conclude
that all $v_i$'s, for $i=1,2,..., k$, are black. By part (b) of
the above lemma, we conclude that the curvature of $d_m$ at $v$
tends to $2\pi$. This shows that $F(u^{(m)})$ tends to infinity of
$Q$. Therefore $F(P)=Q$.

\subsection{Injectivity of $F$ }

The proof uses a variational principle developed in \cite{luo}.
Recall that the map $\mathbf F: \R^n \to \R^n$ is the discrete
curvature map $K_{\A^{-1}(p, w(u))}$ given by (\ref{curv}). Since
$\A$ is a $C^1$ diffeomorphism and the discrete curvature $K:
\T_{pl}(S,V) \to \R^V$ is real analytic, hence the map $\mathbf F$
is $C^1$ smooth. Let $\T_i$, $i=1,..., k$, be the set of all
triangulations so that $(\{p\} \times \R^n) \cap D(\T_i) \neq
\emptyset$ and $\{p\} \times \R^n \subset \cup_{i=1}^k D(\T_i)$.

\begin{lemma}  Let $\phi: \R^n \to \{p\}\times \R^n$ be $\phi(x)=(p, x)$ and
 $U_i =\phi^{-1}((\{p\} \times \R^n) \cap D(\T_i)) \subset \R^n$ and $J =\{ i | $
\text{$int(U_i) \neq \emptyset$}\}. Then $ \R^n =\cup_{i \in J}
U_i$ and $U_i$ is real analytic diffeomorphic to a convex polytope
in $ \R^n$.
\end{lemma}
\begin{proof} By definition, both $\{p\} \times \R^n$ and $D(\T_i)$
are closed and semi algebraic in $T_D(\Sigma)$. Therefore $U_i$ is
closed and semi-algebraic. Now by definition, $X: =\cup_{i \in J}
U_i$ is a closed subset of $ \R^n$ since $U_i$ is closed. If $X
\neq \R^n$, then the complement $ \R^n -X$ is a non-empty open set
which is  a finite union of real algebraic sets of dimension less
than $n$. This is impossible.

Finally, we will show that for any triangulation $\T$ of $(S,V)$
and $p \in T(\Sigma)$, the intersection $U =\phi^{-1}((\{p\}
\times R^n) \cap D(\T))$ is real analytically diffeomorphic to a
convex polytope in a Euclidean space. In fact $\Psi_{\T}^{-1}(U)
\subset \R^{E(\T)}$ is  real analytically diffeomorphic to a
convex polytope. To this end, let $b =\Psi_{\T}^{-1}(p,
(1,1,....,1))$. By definition, $\Psi_{\T}^{-1}(U)$ is give by $$\{
x \in \R^{E(\T)}_{>0}| \exists \lambda \in \R_{>0}^V, x(e) =b(e)
\lambda(v_1) \lambda(v_2),
\partial e=\{v_1, v_2\}, \text{Delaunay
condition (\ref{4}) holds for $x$}\}.$$ We claim that the Delaunay
condition (\ref{4}) consists of linear inequalities in the
variable $\delta: V\to \R_{>0}$ where $\delta(v) =
\lambda(v)^{-2}$. Indeed, suppose the two triangles adjacent to
the edge $e=(v_1, v_2)$ have vertices $v_1, v_2, v_3$ and $v_1,
v_2, v_4$ as shown in figure~\ref{figure_2}. Let $x_{ij}$ (respectively
$b_{ij}$) be the value of $x$ (respectively $b$) at the edge
joining $v_i, v_j$, and $\lambda_i=\lambda(v_i)$. By definition,
$x_{ij}=b_{ij} \lambda_i \lambda_j$. The Delaunay condition
(\ref{4}) at the edge $e=(v_1v_2)$ says that
\begin{equation} \frac{x_{12}^2}{x_{31}x_{32}} +
\frac{x_{12}^2}{x_{41}x_{42}} \leq \frac{x_{31}}{x_{32}} +
\frac{x_{32}}{x_{31}} +\frac{x_{41}}{x_{42}}
+\frac{x_{42}}{x_{41}} \end{equation} It is the same as, using
$x_{ij}=b_{ij}\lambda_i \lambda_j$,

$$ c_3 \frac{\lambda_1 \lambda_2}{\lambda_3^2} + c_4 \frac{\lambda_1\lambda_2}{\lambda_4^2} \leq
c_1\frac{\lambda_2}{\lambda_1} + c_2\frac{\lambda_1}{\lambda_2},$$
where $c_i$ is some constant depending only on $b_{jk}$'s.
Dividing above inequality by $\lambda_1 \lambda_2$ and using
$\delta_i=\lambda_i^{-2}$, we obtain
\begin{equation}\label{1233} c_3 \delta_3 + c_4 \delta_4 \leq c_1 \delta_1 + c_2
\delta_2 \end{equation} at each edge $e\in E(\T)$. This shows for
$b$ fixed, the set of all possible values of $\delta$ form a
convex polytope $Q$ defined by (\ref{1233}) at all edges and
$\delta(v)>0$ at all $v \in V$. On the other hand, by definition,
the map from  $Q$ to $\Psi_{\T}^{-1}(U)$ sending $\delta$ to
$x=x(\delta)$ given by $x(vv') =\frac{b(vv')}{\sqrt{ \delta(v)
\delta(v')}}$ is a real analytic diffeomorphism. Thus the result
follows. $\square$
\end{proof}

Write $\mathbf F=(F_1, ..., F_n)$ which is $C^1$ smooth. By
theorems 1.2 and 2.1 of \cite{luo}, one sees  that (a)
$F_j|_{U_h}$ is real analytic so that $\frac{\partial
F_i}{\partial u_j} =\frac{\partial F_j}{\partial u_i}$ in $U_h$
for all $h \in J$ and (b) the Hessian matrix $[\frac{\partial
F_i}{\partial u_j}]$ is positive semi-definition on each $U_h$ so
that its kernel consists of vectors $\lambda (1,1,...,1)$.
Therefore, the 1-form $\eta =\sum_{i} F_i(u) du_i$ is a $C^1$
smooth 1-form on $\R^n$ so that $d\eta =0$ on each $U_h, h\in J$.
This implies that $d \eta =0$ in $\R^n$. Hence the integral $W(u)
=\int_{0}^u \eta$ is a well defined $C^2$ smooth function on
$\R^n$ so that its Hessian matrix is positive semi-definition.
Therefore, $W$ is convex in $\R^n$ so that its gradient
$\bigtriangledown W=\mathbf F$. Furthermore, since the kernel of
the Hessian of $W$ consists of diagonal vectors $\lambda(1,1,...,
1)$ at each point in $U_h, h\in J$ and $\R^n =\cup_{h \in J} U_h$,
the Hessian of the function $W|_{P}$ is positive definite. Hence
$W|_P$ is strictly convex. Now we use the following well known
lemma,

\begin{lemma} If $W: \Omega \to \R$ is a  $C^1$-smooth strictly convex
function on an open convex set $\Omega \subset \R^m$, then its
gradient $\bigtriangledown W: \Omega \to \R^m$ is an embedding.

\end{lemma}

Apply the lemma to $W|_P$ and use $\bigtriangledown (W|_P)=F$, we
conclude that $F: P \to Q$ is injective.

The discrete Yamabe flow with surgery is the gradient flow of the
strictly convex function $W(u) -\sum_{i=1}^nK^*_i u_i$ which has a
unique minimal point in $P$. In the formal notation, the flow
takes the form $\frac{du_i(t)}{dt} = K_i-K^*_i$ and $u(0)=0$. The
exponential convergence of the flow was established in theorem 1.4
of \cite{luo}.

\section{A Conjecture}

We conjecture that the number of surgery operations used in the
discrete Yamabe flow to find the target PL metric is finite, i.e.,
along the integral curve of the gradient flow of the function
$W(u)-\sum_{i=1}^n K^*_i u_i$, only finitely many diagonal
switches occur. This is supported by our numerical experiments.




There should be a related theory of discrete conformal maps
associated to the discrete Riemann surfaces introduced in this
paper. See \cite{st} for the corresponding discrete conformal maps
for circle packing.

 \noindent David Gu, Department of Computer Science,
Stony Brook University, New York 11794, USA. \newline Email:
gu$@$cs.stonybrook.edu.

\noindent Feng Luo, Department of Mathematics, Rutgers University,
New Brunswick, NJ 08854, USA. \newline Email:
fluo$@$math.rutgers.edu

\noindent Jian Sun, Mathematical Sciences Center, Tsinghua
University, Beijing 100084, China. \newline Email:
jsun$@$math.tsinghua.edu.cn.

\noindent Tianqi Wu, Mathematical Sciences Center, Tsinghua
University, Beijing 100084, China. \newline Email:
mike890505$@$gmail.com.

\bigskip

\bigskip

\noindent {\bf Appendix: A proof of Akiyoshi's theorem}

\medskip
For completeness, we present our proof in this appendix. The
theorem and the proof hold for decorated finite volume hyperbolic
manifolds of any dimension. We state the 2-dimensional case for
simplicity.

\begin{theorem}(Akiyoshi \cite{ak})\label{dau} For a finite area complete
hyperbolic metric $d$ on $\Sigma$, there exist triangulations
$\T_1$, ..., $\T_k$ so that
 for any $w \in R_{>0}^n$, any Delaunay
triangulation of $(d, w)$ is isotopic
$\T_i$, $i\in\{1,2,...,k\}$. 
\end{theorem}
\begin{proof}   We begin by study the shortest geodesics in a
complete finite area hyperbolic surface $(\Sigma, d)$. Recall the
Shimizu lemma \cite{beardon} which  implies that if $w \in (0,
1)^n$, then the associated horoballs $H_i(w)$ in the decorated
metric $(d,w)$ are embedded and pairwise disjoint. Let us assume
without loss of generality that $w \in (0,1)^n$. A geodesic
$\alpha$ from cusp $v_i$ to $v_j$ in $(\Sigma, d)$ is called a \it
shortest geodesic \rm from $v_i$ to $v_j$ if there exists a $w \in
(0, 1)^n$ so that $\alpha \cap X_w$ is a shortest path among all
homotopically non-trivial paths in $X_w$ joining $\partial H_i(w)$
to $\partial H_j(w)$. The shortest property implies that $\alpha
\cap X_w$ is an orthogeodesic. Furthermore, by lemma \ref{l1}, if
$\alpha$ is a shortest geodesic, then for any $w' \in (0, 1)^n$,
$\alpha \cap X_{w'}$ is again a shortest geodesic in $X_{w'}$ from
$\partial H_i(w')$ to $\partial H_j(w')$, i.e., being a shortest
geodesic from $v_i$ to $v_j$ is independent of the choice of
decorations. Indeed, for any geodesic $\beta$ from cusp $v_i$ to
$v_j$, we have
\begin{equation}\label{124} l(\beta \cap X_{w'}) =l(\beta \cap X_w)
-\ln(w_i')-\ln(w_j')+\ln(w_i)+\ln(w_j) \end{equation}

\begin{lemma}\label{shortgeo} Suppose $(\Sigma, d)$ is a finite area complete
hyperbolic surface. Then

(a) there are only finitely many shortest geodesics from $v_i$ to
$v_j$.

(b) there is $\delta_{ij} =\delta_{ij}(\Sigma, d) >0$ so that if
$\alpha$ is a shortest geodesic from $v_i$ to $v_j$ and $\beta$ is
another geodesic from $v_i$ to $v_j$ with $|l( \beta \cap X_w)
-l(\alpha \cap X_w)| \leq \delta_{ij}$, then $\beta$ is a shortest
geodesic.

(c) given $v_i$, if $\alpha$ is a shortest orthogeodesic geodesics
among all orthogeodesics in $X_w$ from $\partial H_i$ to $\partial
X_w$, then $\alpha^*$, the complete geodesic containing $\alpha$,
is an edge of the decorated metric $(d,w)$ and the mid-point of
$\alpha$ is in $R_j(w)$.

 \end{lemma}
\begin{proof} The first part follows from the simple fact that on any
compact surface $X_w$, for any constant $C$, there are only
finitely many orthogeodesics of length at most $C$.   Part (b)
follows from (a) and equality (\ref{124}). Part (c) follows from
the definition of Voronoi cells and its dual. Note that in
general, if $\beta$ is a shortest orthogeodesic in $X_w$ between
$\partial H_i(w)$ and $\partial H_j(w)$, $\beta^*$ may not be an
edge in any Delaunay triangulation of $(d,w)$. $\square$
\end{proof}

Now we prove the theorem by contradiction. Suppose otherwise,
there exists a sequence of decorated metrics $(d, w^{(m)})$ where
$w^{(m)} =(w^{(m)}_1, ..., w^{(m)}_n) \in \R^n$ so that the
associated Delaunay triangulations $\T_m =\T(d, w^{(m)})$ are
pairwise distinct in $(\Sigma, d)$.  After normalizing $w^{(m)}$
by scaling, relabel the vertices $v_1, ..., v_n$ and taking
subsequences, we may assume

(i) $w^{(m)}_1 =\max\{ w^{(m)}_i | i=1,2,...,n\} =1/2$;

(ii) for each $i=1,2,..., n$, the limit $\lim_{m} w^{(m)}_i =t_i
\in[0, 1/2]$ exists;

(iii) $t_1,..., t_k >0$ and $t_{k+1}=...=t_n=0$.

For simplicity, we use $E_{ij}(\T)$ to denote the subset of all
edges of $\T$ joining $v_i$ to $v_j$. We will derive a
contradiction by showing that $\cup_{m} E_{ij}(\T_m)$ is a finite
set.

\begin{lemma}  There exists a constant $C>0$ so that for all $i,
j \leq k$, and all $e \in E_{ij}(\T_m)$, the length
$$ l( e \cap X_{w^{(m)}}) \leq C.$$
In particular, $\cup_{m} E_{ij}(\T_m)$ is a finite set.
\end{lemma}

\begin{proof} For any $\delta \in (0, 1/2)$, let
 $u^{(m)}(\delta) = (w^{(m)}_1, ..., w^{(m)}_k, \delta, ..., \delta) \in \R^n$.
  Fix a $\delta$, since $\lim_{m} w^{(m)}_{j} =0$ for $j>k$, for $m$ large,
 each point $x \in
X_{u^{(m)}(\delta)}$ is in some Voronoi cell $R_i(w^{(m)})$ for
some $i \leq k$. Therefore, there is a small $\delta >0$ so that
for all $i, j =1,2,..., k$, all large $m$, and all $e \in
E_{ij}(\T_m)$, $ e \cap X_{w^{(m)}} \subset X_{u^{(m)}(\delta)}$.
 By the assumption that $t_1,
..., t_k
>0$ and by choosing $\delta$ smaller than $\min\{t_1, ..., t_k\}$, we
see that the surface $X_{u^{(m)}(\delta)}$ is a subset of the
compact surface $X_{(\delta, ..., \delta)}$. Therefore, there is a
constant $C>0$ so that $diam(X_{u^{(m)}(\delta)}) \leq C/2$ for
all $m$.   Note  if $e \in E(\T(d,w))$ is an edge, then the length
of the orthogeodesic $e \cap X_w$ in metric $d$ satisfies,
\begin{equation}\label{orthlength}
 l(e \cap X_w) \leq 2 diam(X_w),
\end{equation} where $diam(Y)$ is the diameter of a metric space
$Y$. Indeed,  $l(e \cap X_w) \leq diam(R_i(w)) + diam(R_j(w)) \leq
2 diam(X_w)$.
 This shows, by (\ref{orthlength}),  that
$$ l( e \cap X_{w^{(m)}}) \leq  l(e \cap X_{u^{(m)}(\delta)}) \leq 2 diam(X_{u^{(m)}(\delta)}) \leq
C.$$

Finally, since for any constant $C$, there are only finitely many
orthogeodesics in $X_{(\delta, ..., \delta)}$ of lengths at most
$C$, it follows that  $\cup_{m} E_{ij}(\T_m)$ is finite.
 $\square$
\end{proof}

Now   for $m$ large,  each point in $X_{u^{(m)}(1/2)}$ is in
$\cup_{i=1}^k R_i(w^{(m)})$. Therefore for large $m$, if $i,j
>k$, then $E_{ij}(\T_m) =\emptyset$ since an edge $e \in
E_{ij}(\T_m)$ must intersect $X_{u^{(m)}(1/2)}$. Hence if
$E_{jh}(\T_m) \neq \emptyset$, then $h \leq k$.

\begin{lemma}  There is  $n_0$ so that if $m \geq n_0$, $j >k$ and
$e \in E_{ij}(\T_m)$, then $e$ is a shortest geodesic from $v_i$
to $v_j$.   In particular for $j>k$ and $i \leq k$, the set
$\cup_{m}E_{ij}(\T_m)$ is finite.
\end{lemma}

\begin{proof} We need to study the Voronoi cell $R_j(w^{(m)})$. Since
$\lim_{m} w^{(m)}_j =0$ and $t_i >0$, for large $m$,  the Voronoi
cell $R_j(w^{(m)}) \subset H_j(u^{(m)}(1/2))$. Let $\partial_0
R_j(w^{(m)})$ be the piecewise geodesic boundary component
$\partial R_j(w^{(m)}) -\partial H_j(w^{(m)})$.

\noindent {\bf Claim} 
for any two edges $a_m, b_m$ in $\partial_0 R_j(w^{(m)})$,
\begin{equation}\label{3456} \lim_{m}|dist(a_m, H_j(w^{(m)})) -dist(b_m, H_j(w^{(m)}))|= 0. \end{equation}

Assuming the claim, we finish the proof of the lemma as follows.
 Let $\epsilon_m$ be a shortest orthogeodesic in
$X_w$ from $\partial H_j(w^{(m)})$ to $\partial X_w$ and
$e_m'=\epsilon_m^*$ be the complete geodesic containing
$\epsilon_m$. Then by lemma \ref{shortgeo}, $e_m' \in \cup_{i=1}^n
E_{ij}(\T_m)$. Let the dual of $e_m'$ be the edge $a_m$ of
$\partial_0
R_j(w^{(m)})$. 
 For any
edge $e_m \in E_{ij}(w^{(m)})$ dual to an edge $b_m$ of
$\partial_0 R_j(w^{(m)})$, we have $l(e_m \cap X_{w^{(m)}})=2
dist(b_m, H_j(w^{(m)}))$ by the definition of Delaunay. Therefore
by (\ref{3456})
$$\lim_{m}|l(e_m \cap X_{w^{(m)}}) -l(e_m' \cap X_{w^{(m)}})|=0.$$ By
lemma \ref{shortgeo}, since $e_m'$ is a shortest geodesic, $e_m$
is also a shortest geodesic for $m$ large.

To see the clam (\ref{3456}),  recall that a simple geodesic loop
on $(\Sigma, d)$ is a smooth map $\alpha: [0,1] \to \Sigma$ so
that $\alpha(0)=\alpha(1)$, $\alpha|_{(0,1)}$ is a geodesic and
$\alpha|_{[0,1)}$ is injective. 
Now for each $i \leq k$ and for $m$ large, the equidistance curve
$\alpha_{i,j}(m)$ between $H_i(w^{(m)})$ and $H_j(w^{(m)})$ is a
simple geodesic loop in the cusp region  $H_j(s_m(1,1,...,1))$
where $\lim_m s_m =0$.  This is due to the fact that $w_j^{(m)}
\to 0$ and $w_i^{(m)} \to t_i
>0$.  It is well known
that if $\alpha$ is a simple geodesic loop in a cusp region
$H_j(w)$, then the length of $\alpha$ is less than $w_j$.
Therefore,
  $l(\alpha_{i,j}(m)) \leq  s_m$ and $\lim_m l(\alpha_{i,j}(m)) =0$.
By definition, the boundary $\partial_0 R_j(w^{(m)}) \subset
\cup_{i} \alpha_{i,j}(m)$. If $a_m, b_m$ are two edges $\partial_0
R_j(w^{(m)}$, then by definition $|dist(a_m, H_j(w^{(m)}))
-dist(b_m, H_j(w^{(m)}))| \leq \sum_{i=1}^k
l(\alpha_{i,j}(m))$. Therefore (\ref{3456}) follows from $\lim_m
l(\alpha_{i,j}(m)) =0$ . $\square$

\end{proof}

\end{proof}

\end{document}